\documentclass[final, nomarks]{dmtcs-episciences}
\usepackage{amsmath,amssymb,amsthm}
\usepackage{courier}
\usepackage{tikz}
\usetikzlibrary{calc,matrix,arrows,decorations.markings}
\usepackage{array}
\usepackage{color}
\usepackage{enumerate}
\usepackage{nicefrac}
\usepackage{listings}
\usepackage[utf8]{inputenc}
\usepackage[english]{babel}
\usepackage{blindtext}
\newtheorem{theorem}{Theorem}
\newtheorem{lemma}[theorem]{Lemma}
\newtheorem{conjecture}{Conjecture}
\newtheorem{definition}{Definition}
\usepackage[utf8]{inputenc}
\usepackage{subfigure}
\usepackage[round]{natbib}

\author{Ryan Alweiss\affiliation{Massachusetts Institute of Technology}\thanks{The author was supported by NSF grant NSF-1358659 and NSA grant H98230-16-1-0026.}}

\title{Asymptotic results on Klazar set partition avoidance}

\begin{document}

\received{2017-2-2}

\revised{2017-8-19, 2017-12-22, 2018-3-13}

\accepted{2018-3-13}

\publicationdetails{19}{2018}{2}{7}{3112}

\maketitle

\begin{abstract}  We establish asymptotic bounds for the number of partitions of $[n]$ avoiding a given partition in Klazar's sense, obtaining the correct answer to within an exponential for the block case.  This technique also enables us to establish a general lower bound.  Additionally, we consider a graph theoretic restatement of partition avoidance problems, and propose several conjectures. \end{abstract}

\section{Introduction}

The notion of partitions of $[n]=\{1,2,\dots,n\}$ avoiding a partition $\tau$ were introduced by Klazar in 2000 \citep{klazar}.  Much of the work with such partitions has been explicit and enumerative in nature.  See, for instance, Bloom and Saracino \citep{bloom} and Sagan \citep{sagan}.  They have given inequalities between the number of partitions of $[n]$ avoiding partitions $\tau_1$ and $\tau_2$ using injections or bijections, or have explicitly computed these numbers with generating functions.  Furthermore, most of these approaches can only enumerate the number of partitions for a relatively small class of $\tau$.

There have been some asymptotic approaches to the partitions of $[n]$ avoiding $\tau$.  Examples include the original paper where Klazar introduced the concept \citep{klazar} and a later paper he wrote with Marcus \citep{km}.  However, these efforts have focused on very specific $\tau$ for which the results are closely related to pattern-avoiding permutations in general and the Stanley-Wilf conjecture in particular. In this paper, we examine the growth rate of the number of partitions of $[n]$ avoiding a partition $\tau$ for a different class of $\tau$ through asymptotic techniques.

Define a \textit{set partition} of $[n]$ to be a set of disjoint subsets of $[n]$, $\{B_1, B_2, \ldots , B_r\}$, whose union is $[n]$.  The $B_i$ are called \textit{blocks}.  The order of the elements of each $B_i$ does not matter, nor does the order in which the $B_i$ are written.  Canonically, we arrange the blocks by increasing order of the minimal element in each block, and write the elements of each block in ascending order.  We write the blocks separated by slashes, so $134/25$ describes the partition of $[5]$ where one set is $\{1,3,4\}$ and the other set is $\{2,5\}$.

\begin{definition}

A partition $\sigma$ of $n$ \emph{contains} a partition $\tau$ of $k$ if there exists a subpartition $\sigma'$ of $\sigma$ that has the same relative order as $\tau$.  Otherwise, $\sigma$ \emph{avoids} $\tau$.

\end{definition}

By this we mean there is an increasing bijective map from $\sigma'$ to $\tau$ that preserves blocks.  For instance, $\sigma=124/35$ contains $1/23$, because $\sigma'=1/35$ is a subpartition of $\sigma$.  Here $\sigma'$ has the same relative order as $\tau$, with the bijective map $f(1)=1$, $f(3)=2$, $f(5)=3$.  

\begin{definition}

We use $A_n(\tau)$ to denote the number of set partitions of $[n]$ avoiding a fixed partition $\tau$.

\end{definition}
In this paper, we are primarily concerned with the function $A_n(\tau)$.  Bloom and Saracino \citep{bloom} have studied the behavior of $A_n(\tau)$ through injective mappings, proving results of the form $A_n(\tau_1) \le A_n(\tau_2)$ for various choices of $\tau_1$ and $\tau_2$.  Along these lines, they conjecture that the partition $12 \cdots k$ of a single block is the easiest to avoid. 

\begin{conjecture}
\label{blockmax} \emph{(Bloom, Saracino \citep{bloom})}

If $\tau$ is a set partition of $[k]$ other than $(12 \cdots k)$ then $$A_n(12\cdots k)\geq A_n(\tau)$$ with strict inequality for $n$ sufficiently large.

\end{conjecture}

In general $A_n(\tau)$ is very large.  In order to better get a handle on it, let $F_n(\tau)=\frac{\log(A_n(\tau))}{n\log(n)}$.  If $A_n(\tau)=n^{bn}$, then $F_n(\tau)=b$.  We are interested in the asymptotic behavior of $F_n(\tau)$ for large $n$, and for partitions $\tau$.  Our interest in $F_n$ stems from the observation that if $F_n(\tau_1)>F_n(\tau_2)$ for large enough $n$, where $\tau_i$ are partitions of $[k]$, then $A_n(\tau_1)>A_n(\tau_2)$ for large enough $n$.  There is one known result about $F_n(\tau)$, due to Klazar and Marcus \citep{km}, which we state here as a lemma.

\begin{lemma}
\label{exponential}
If $\tau$ avoids $123$ and $12/34$, then $A_n(\tau)$ grows exponentially and so $F_n(\tau)=0$.
\end{lemma}

We first prove that for every such $\tau$ there exists a permutation $\sigma$ of $[k]$ so that $\tau$ is contained in a set partition of $[2k]$ given by $\sigma$.  Specifically, $\{r,k+\sigma(r)\}$ for $1 \le r \le k$ will be exactly the blocks of this partition.  To prove this, note that all blocks of $\tau$ have size one or two by the $123$ avoidance, and furthermore there exists a half integer $c$ so that all blocks of size two have one element greater than $c$ and the other element less than $c$, by the $12/34$ avoidance.  Now, we modify $\tau$ by taking the blocks of one element which is less than $c$ and appending an unused number greater than $c$, and taking the blocks with one element which is greater than $c$ and appending an unused number less than $c$ (possibly negative).  The result has all blocks of size $2$, and has the same relative order as a set partition of $2k$ where $\{r,k+\sigma(r)\}$ are the blocks for $1 \le r \le k$, for some $k$.  The elements less than $c$ correspond to $1 \le i \le k$ and the elements greater than $c$ correspond to $k+1 \le i \le 2k$ in the relative ordering.  Call a set partition whose blocks are $\{r,k+\sigma(r)\}$ a \emph{permutation partition} of $[2k]$, and identify it with $\sigma \in S_k$.  So every $\tau$ which avoids $123$ and $12/34$ is contained in a permutation partition.  Lemma~\ref{exponential} is thus equivalent to showing that only exponentially many partitions of $[n]$ avoid a given fixed permutation partition.

Klazar and Marcus proved that only exponentially many set partitions of $[n]$ avoid a fixed permutation partition; it is the $1$-regular case of their Corollary 2.2 \citep{km}.  Balogh, Bollobas, and Morris independently proved this result as their Theorem 3 \citep{balogh}.  Thus Lemma~\ref{exponential} is established.

A permutation partition corresponding to $\sigma \in S_n$ contains a permutation partition corresponding to $\sigma' \in S_k$ if and only if $\sigma$ contains $\sigma'$.  Here the notion of permutation containment is the usual one; $\sigma$ contains $\sigma'$ if and only if there is a substring of $\sigma$ that has the same relative order as $\sigma'$.  In particular, at most exponentially many set partitions of $[2n]$ avoid a fixed permutation partition of $[2k]$ by Lemma \ref{exponential}, so at most exponentially many permutation partitions of $[2n]$ avoid a fixed permutation partition of $[2k]$.  It follows that at most exponentially many permutations of $[n]$ avoid a fixed permutation of $[k]$.  Thus, Lemma~\ref{exponential} actually implies the well known Stanley-Wilf conjecture first proved by Marcus and Tardos \citep{stanley}.  In fact, the implication of the Stanley-Wilf conjecture was first noticed in 2000 by Klazar in the paper where he developed pattern avoidance in set partitions \citep{klazar}. 

\begin{definition}
A \emph{layered} partition is a partition whose blocks consist of sets of consecutive integers.
\end{definition}

For example, $12/345/67$ is layered, whereas $13/245/67$ is not.

In this paper, we compute the limit of $F_n(\tau)$ for any layered partition $\tau$.  We prove the following theorem.

\begin{theorem}
\label{layered}
If a layered partition $\tau$ of $[k]$ is composed of $r$ blocks, with $k>r$, then  $$A_n=\Theta(1)^nn^{n(1-\frac{1}{k-r})}$$ and so $$\lim_{n \to \infty}F_n(\tau)=1-\frac{1}{k-r}.$$

\end{theorem}

From this, the converse of Lemma~\ref{exponential}, which was already known to Klazar and Marcus, follows. If $\tau$ contains $123$ or $12/34$ then $A_n(\tau) \ge \min(A_n(123),A_n(12/34))$ and $F_n(\tau) \ge \min(F_n(123),F_n(12/34)) \ge \frac{1}{2}-o(1)$. 

We conjecture (among other things) that $F_n(\tau)$ tends to $1-\frac{1}{g(\tau)}$ where $g(\tau)$ is some integer that is maximized for fixed $k$ exactly when $\tau$ is a single block.  This is essentially an asymptotic form of Conjecture~\ref{blockmax}.

\section{Upper bounds on $A_n$ for a single block}

We use recursion to prove an upper bound on the number of partitions of $[n]$ that avoid the block partition of $[k]$.  This is a fairly well studied sequence \citep{oeis}, but we only need rough bounds here.

Fix $k \ge 2$, and let $L_k=12 \cdots k$ be the partition of a single block.  We let $f(n)=A_n(L_k)$ in this section.  Note the partitions counted by $f(n)$ are exactly those with all blocks of size at most $k-1$.

We have $f(n+1)=\displaystyle \sum_{i=0}^{k-2} \dbinom{n}{i}f(n-i)$, because if the block containing $1$ has $i \le k-2$ other elements, there are $\dbinom{n}{i}$ choices for these other elements, and then $f(n-i)$ ways to partition the remaining elements.

We will show that $$f(n) \le k^n(n^{n\left(1-\frac{1}{k-1}\right)}).$$  We do so by induction.  When $1 \le n \le k$, $f(n) \le n^n \le k^n$, so the bound obviously holds.  This establishes our base case.
\begin{flushleft}
For the inductive step, note we have $$f(n+1) \le \displaystyle \sum_{i=0}^{k-2} n^if(n-i) \le \displaystyle \sum_{i=0}^{k-2} n^ik^{n-i}n^{(n-i)(1-\frac{1}{k-1})}$$ $$\le \displaystyle \sum_{i=0}^{k-2} n^ik^n n^{(n-i)(1-\frac{1}{k-1})}=k^n \displaystyle \sum_{i=0}^{k-2} n^{i+(n-i)(1-\frac{1}{k-1})}$$ by the inductive hypothesis, and because $n-i \le n$.  We have $i+(n-i)(1-\frac{1}{k-1})=n-(\frac{1}{k-1})(n-i)$ is an increasing function of $i$, so the $i=k-2$ term represents the largest term in this sum, and it is $(k-2)+(n-(k-2))(1-\frac{1}{k-1})=(n+1)(1-\frac{1}{k-1})$.
\end{flushleft}
Hence, $$f(n+1) \le k \cdot k^n n^{(n+1)(1-\frac{1}{k-1})} \le k^{n+1} (n+1)^{(n+1)(1-\frac{1}{k-1})}$$ completing the proof.

\section{Upper bounds on $A_n$ in the layered case}

In this section, we let $L_{a_1, \cdots, a_r}$ be the layered partition where the smallest $a_1$ elements are in a block, the $a_2$ next smallest are in a block, and so on, with $\displaystyle \sum a_i=k$.  This partition can also be written $12 \cdots a_1 / (a_1+1) \cdots (a_1+a_2) / \cdots / \cdots (a_1+\cdots +a_r)$.  In particular the block partition dealt with in the previous section is $L_k$.

\begin{lemma}
\label{layercontain}
Any partition of $[n]$ with $r$ blocks of size at least $k-r+1$ contains $L_{a_1, \cdots, a_r}$.

\end{lemma}

\begin{proof}

We assume $r>1$, because the $r=1$ case is trivial.

Say a partition has $r$ such blocks $A_1, A_2, \ldots, A_r$.  We assume without loss of generality that for any $j=1,2, \ldots, r$, $A_j$ has the minimum $a_1+ \cdots +a_j-(j-1)$th smallest element out of $\{A_j, A_{j+1}, \ldots, A_r\}$, by way of an algorithm.  Note that $a_1+ \cdots +a_j-(j-1) \le a_1+ \cdots +a_r-(r-1)=k-r+1$, so indeed it makes sense to speak of the $a_1+ \cdots +a_j-(j-1)$th smallest element of each block $A_j, A_{j+1}, \ldots, A_r$.  Of course, reordering the blocks is permissible, as it does not affect the set partition.  It is just for notational convenience.

We now describe the algorithm.  For $1 \le j \le r$, on the $j$th step we relabel $\{A_j, A_{j+1}, \ldots, A_r\}$ in some way so $A_j$ has the minimum $a_1+ \cdots +a_j-(j-1)$th smallest element out of $\{A_j, A_{j+1}, \ldots, A_r\}$.  Since immediately after the $j$th step of the algorithm, $A_j$ has the minimum $a_1+ \cdots +a_j-(j-1)$th smallest element out of $\{A_j, A_{j+1}, \ldots, A_r\}$, and because in later steps only $\{A_{j+1}, \ldots, A_r\}$ are permuted among themselves, this property is preserved.

Let the interval $S_j$ consist of the $a_1+ \cdots +a_{j-1}-(j-2)$th through $a_1+ \cdots +a_j-(j-1)$th smallest elements of $A_j$, inclusive.  The set $S_j$ has precisely $a_j$ elements.

Let $S$ be the union of the $S_j$.  Each $S_j$ is in its own block.  Also, the largest element of $S_j$ is smaller than the smallest element of $S_{j+1}$, because $A_j$ has a smaller $a_1+\cdots +a_j-(j-1)$th smallest element than $A_{j+1}$ does by construction.  As such, $S$ has the same relative order as $L_{a_1, \cdots, a_r}$.

\end{proof}

It follows that any partition of $[n]$ which avoids $L_{a_1, \cdots, a_r}$ has at most $r-1$ blocks of size at least $k-r+1$.  Thus, it suffices to bound from above the number of these partitions.

\begin{theorem}
\label{blockbound}
The number of partitions of $[n]$ with at most $r-1$ blocks of size at least $k-r+1$ is bounded by $(\frac{k+1}{2})^{2n}n^{n\left(1-\frac{1}{k-r}\right)}$.

\end{theorem}

\begin{proof}

Call a block of size at least $k-r+1$ a \emph{big} block.  There are at most $r-1$ big blocks.

For each element of $[n]$, we first decide whether it is in a block of size at most $k-r$, or in the $j$th big block, for some $1 \le j \le r-1$.  There are at most $r^n$ choices for this initial stage.

Afterward, this uniquely sorts some set of $m$ of the elements $[n]$ into \text{big} blocks.  There are $$A_{n-m}(12 \cdots (k-r+1)) \le A_n(12 \cdots (k-r+1)) \le (k-r+1)^nn^{n\left(1-\frac{1}{k-r}\right)}$$ ways to partition the remaining elements into blocks of size at most $k-r$.

This yields a bound of $$r^n(k-r+1)^nn^{n\left(1-\frac{1}{k-r}\right)} \le \left(\frac{k+1}{2}\right)^{2n}n^{n\left(1-\frac{1}{k-r}\right)}$$ as desired.

\end{proof}

Thus, if we have $f(n)=A_n(L_{a_1, \cdots, a_r})$, then $f(n) \le (\frac{k+1}{2})^{2n}n^{n\left(1-\frac{1}{k-r}\right)}=O_k(1)^nn^{n\left(1-\frac{1}{k-r}\right)}$.

\section{Lower bounds on $A_n$ in the layered case}

Again, let $f(n)=A_n(L_{a_1, \cdots, a_r})$, where the notation for layered partitions is as in the previous section.  Here $\sum a_i=k$ and we are counting the number of partitions of $[n]$ avoiding a fixed layered partition of $[k]$.

\begin{theorem}

We have $f(n) \ge c^nn^{n\left(1-\frac{1}{k-r}\right)}$, where $c=\Omega_k(1)$ and $k>r$.

\end{theorem}

\begin{proof}

It suffices to prove the result in the case that $n$ is divisible by $k-r$, because $c$ can easily be adjusted.

Call a partition of $[n]$ \emph{uniform} if it is composed of $\frac{n}{k-r}$ blocks of size $k-r$, and for every $0 \le j < k-r$ it is true that $S_j=[j(\frac{n}{k-r})+1,(j+1)(\frac{n}{k-r})]$ contains exactly one element from each block.

First, we count the number of uniform partitions.  Every block of the partition has exactly one element in each $S_j$, so for every $0 < j<k-r$ we can choose some way to match the $\frac{n}{k-r}$ elements $[j(\frac{n}{k-r})+1,(j+1)(\frac{n}{k-r})]$ with the $\frac{n}{k-r}$ blocks.  This yields a total of 

$$\left(\frac{n}{k-r}\right)!^{k-r-1}=\left(\Omega_k(1)^n n^{\frac{n}{k-r}}\right)^{k-r-1}=\Omega_k(1)^n n^{n\left(1-\frac{1}{k-r}\right)}$$

\begin{flushleft}
uniform partitions, using Stirling's approximation.
\end{flushleft}

Now we prove that these uniform partitions avoid $L_{a_1, \cdots, a_r}$, for some fixed $(a_1, \cdots, a_r)$.  Assume such a partition contained $L_{a_1, \cdots, a_r}$.  Equivalently, this partition contains some $x_i$ with $x_1<x_2< \cdots <x_k$ so that the smallest $a_j$ elements which were not in the smallest $a_1+ \cdots +a_{j-1}$ of these elements are their own block.  Consider pairs $(x_i,x_{i+1})$ in this partition.  If $x_i$ and $x_{i+1}$ are in the same block, they must be in different $S_j$.  Since there are only $k-r$ choices for $S_j$, if we consider the sequence $x_1<x_2< \cdots <x_k$, there are at most $k-r-1$ numbers with $1 \le i \le k-1$ so that $x_i$ and $x_{i+1}$ are in different $S_j$.  Hence, there are at most $k-r-1$ numbers $i \in [1,k-1]$ so that $x_i$ and $x_{i+1}$ are in the same block, and at least $r$ numbers $i \in [1,k-1]$ so that $x_i$ and $x_{i+1}$ are in different blocks.  This is a contradiction, because $L_{a_1, \cdots, a_r}$ has only $r-1$ such numbers.  \end{proof}

Note that the above proof works when $L_{a_1, \cdots, a_r}$ is replaced by any partition of $[k]$ such that there are at most $r$ choices for $i$ such that $1 \le i \le k-1$ and $(i,i+1)$ are in different blocks of that partition.

Clearly it is easier to avoid a partition than to avoid a pattern it strictly contains.

Thus, a partition of $[k]$ that contains a partition $P$ of $[k']$ with less than $r'$ choices for $i$ such that $1 \le i \le k'-1$ and $(i,i+1)$ are in different blocks of $P$ has $\Omega_{k'}(1)^nn^{n\left(1-\frac{1}{k'-r'}\right)}=\Omega_{k}(1)^nn^{n\left(1-\frac{1}{k'-r'}\right)}$ partitions of $[n]$ avoiding it as well.  This represents a nontrivial lower bound for the general case.

Combining all these results, we have $$A_n(L_{a_1, \cdots, a_r})=\Theta(1)^nn^{n(1-\frac{1}{k-r})}$$ whenever $k-r \ge 1$.  For the $r=k$ case, the singleton subcase $\tau=1/2/\cdots /k$, $A_n(1/2/\cdots /k)$ is the number of ways to put $n$ elements in at most $k-1$ urns, so it is at most $(k-1)^{n}$ and grows only exponentially.  This establishes Theorem~\ref{layered}.

Interestingly, Bloom and Saracino \citep{bloom} prove that in the $ij/1/2/ \cdots /k$ case $A_n$ is at most as large as it is in the singleton case via an injection argument, even though we observe that partitions with fewer blocks tend to have larger values of $A_n$. This shows that in fact $A_n$ is also exponential for $ij/1/2/ \cdots /k$, and the $j=i+1$ case of this corresponds to the $r=k-1$ case of our main theorem.

\section{A graph-theoretic restatement}

Another possible direction of research is to try to apply graph theoretic results and techniques to partition avoidance problems.

\begin{definition}

A \emph{complete partite} undirected graph on $n$ vertices is a graph whose complement is the union of vertex-disjoint complete graphs.  

\end{definition}

\begin{definition} A \emph{directed acyclic complete partite graph} (DACP) is a directed acyclic graph whose underlying undirected graph is a complete partite graph.
\end{definition}

We can form a bijection between set partitions of $[n]$ and DACPs with $n$ vertices.  Given a set partition of $[n]$, direct an edge from $a$ to $b$ if and only if $a$ and $b$ are in different blocks with $a>b$.  Clearly this graph lacks directed cycles, and the complement of its underlying undirected graph is the union of cliques corresponding to blocks of the partition.  Hence from any partition of $[n]$ we form a unique DACP on $n$ vertices.

Given a DACP on $n$ vertices, we can likewise reconstruct the set partition it represents.  Call two vertices $v$ and $u$ of a DACP \textit{indistinguishable} if they are not connected, and have the same in-neighborhoods and out-neighborhoods.  This is an equivalence relation, and therefore partitions the $n$ vertices into equivalence classes.  Given two equivalence classses, there must exist a vertex $v$ with an in-edge from all vertices in one class and an out-edge to all vertices in the other.  As such, the DACP induces a total order on the equivalence classes.  Thus given an equivalence class $A$ of size $r$ which is larger than $s$ other elements, the elements of $A$ will be the $(s+1)$st through $(s+r)$th smallest elements for any extension of the partial order of the DACP to a total order.  Hence, any two total orders are isomorphic.  Pick an arbitrary such order, an assignment of the elements of $[n]$ to the vertices of the DACP.  Then this corresponds to a set partition, where the independent sets are the blocks.  The partition is uniquely determined because of the isomorphism property.

Thus, there is a natural bijection between set partitions of $[n]$ and DACPs with $n$ vertices.

Furthermore, if a set partition of $[n]$ contains a set partition of $[k]$ then the former has a subset order isomorphic to the latter, and so the DACP corresponding to the former has the DACP corresponding to the latter as an induced subgraph.  Likewise if a DACP has another DACP as an induced subgraph, this represents a containment of their corresponding partitions.

Thus $A_n(\tau)$ is the number of DACPs on $n$ vertices that avoid some specific DACP on $k$ vertices (corresponding to $\tau$) as a subgraph.

\section{Conjectures}

We propose several conjectures about the behavior of $A_n(\tau)$ and $F_n(\tau)$.

\begin{conjecture} For all $\tau$, $\lim_{n \to \infty}F_n(\tau)=1-\frac{1}{c}$ for some constant $c$ depending on $\tau$.
\end{conjecture}

\begin{conjecture} For all $\tau$, the constant $c$ in Conjecture $2$ is an integer. \end{conjecture}

\begin{conjecture} When $\tau$ is a partition of $[k]$, the constant $c$ in Conjecture 2 satisfies $c \leq k-1$ with equality only for the one-block partition. \end{conjecture}

Note that Conjecture 4 implies Conjecture~\ref{blockmax} in the limit, and is also a strengthening of Conjecture 2.

The next conjecture is a revised version of a conjecture by the author that appeared in an earlier version of this paper.  In fact, it was first conjectured by Gunby \citep{gunby} who built upon the work of the author.  To state it, we first need a definition.

\begin{definition}

Given a set partition $\tau$ of $[n]$, let the \emph{permeability} $pm(\tau)$ be the minimum $k$ such that $[n]$ can be partitioned into $k+1$ intervals (i.e. sets of consecutive integers) and each of these intervals has at most one element from each block of $\tau$.

\end{definition}

Gunby \citep{gunby} showed that $A_n(\tau) \ge 1-\frac{1}{pm(\tau)}$ in general, and that this is a strengthening of Theorem 3.

\begin{conjecture}

$F_n(\tau)=1-\frac{1}{pm(\tau)}$

\end{conjecture}

This conjecture subsumes all of the conjectures before it.  All conjectures are true for $k \le 4$.  The case where $pm(\tau)=1$ is Lemma 1.

After establishing some or all of these conjectures, for various $\tau$ one could obtain more precise asymptotics.  Consider the function $\frac{A_n(\tau)}{n^{n\left(1-\frac{1}{c}\right)}}$.  If it is known that $C^{-n}n^{n\left(1-\frac{1}{c}\right)} \le A_n(\tau) \le C^nn^{n\left(1-\frac{1}{c}\right)}$ for some $C>1$, perhaps the gap between the two multipliers in front of $n^{n\left(1-\frac{1}{c}\right)}$ can be narrowed to something sub-exponential.

We can formulate a conjecture along these lines.

\begin{conjecture}

$|F_n(\tau)-(1-\frac{1}{c})|=O(\frac{1}{\log(n)})$ for the appropriate $c$ so that $\lim_{n \to \infty}F_n(\tau)=1-\frac{1}{c}$ and $c \ge 2$.

\end{conjecture}

This conjecture appears to be true for small cases, and certainly holds for blocks.

A graph theoretic approach may also bear some fruit.  The limit for $F_n$ of the single block, $1-\frac{1}{k-1}$, seems reminiscient of the Erd\H{o}s-Stone theorem \citep{erdos} especially given its restatement as avoidance of an independent set of size $k$, the complement of a graph of chromatic number $k-1$.  In light of the graph theoretic restatement, perhaps there is some meaningful connection.  Note that Klazar and Marcus found the cases where $F_n=0$, proving their generalization of the Stanley-Wilf conjecture, using a different graph theoretic restatement involving undirected graphs \citep{km}.  Their interpretation does not generalize to stating the entire problem graph theoretically, however.

\section{Acknowledgements}

This research was made possible by NSF grant NSF-1358659 and NSA grant H98230-16-1-0026.  It was conducted at the University of Minnesota Duluth REU, under the supervision of Joe Gallian who suggested the problem and helped with the exposition.  Special thanks to Ben Gunby for his helpful comments on the manuscript and noting the connection to the Stanley-Wilf conjecture.

\end{document}